\documentclass[12pt]{amsart}

\usepackage{amsmath}
\usepackage{amssymb}
\usepackage[margin=1in]{geometry}
\usepackage{graphicx}
\usepackage{microtype}

\usepackage{amsthm}
\theoremstyle{plain}
\newtheorem{theorem}{Theorem}
\newtheorem{corollary}[theorem]{Corollary}
\newtheorem{proposition}[theorem]{Proposition}
\newtheorem{lemma}[theorem]{Lemma}

\theoremstyle{definition}
\newtheorem{definition}{Definition}
\newtheorem{question}{Question}
\theoremstyle{remark}
\newtheorem*{remark}{Remark}

\newcommand{\ex}{\text{\sffamily X}}

\newcommand{\F}{\mathbb F}
\newcommand{\Fpb}{\overline{\mathbb F}_p}
\newcommand{\cO}{\mathcal O}
\newcommand{\A}{\mathbb A}
\renewcommand{\P}{\mathbb P}

\newcommand{\Q}{\mathbb Q}
\newcommand{\bp}{\mathbf p}
\newcommand{\bq}{\mathbf q}
\newcommand{\br}{\mathbf r}

\newcommand{\rat}{\dashrightarrow}
\newcommand{\abs}[1]{\left\lvert #1 \right\rvert}
\newcommand{\set}[1]{\left\{ #1 \right\}}

\makeatletter
\def\blfootnote{\gdef\@thefnmark{}\@footnotetext}
\makeatother

\DeclareMathOperator{\Bl}{Bl}
\DeclareMathOperator{\Pic}{Pic}
\DeclareMathOperator{\Aut}{Aut}

\title{Tri-Coble Surfaces and their Automorphisms}
\author{John Lesieutre}

\begin{document}

\begin{abstract}
We construct some positive entropy automorphisms of rational surfaces with no periodic curves.
The surfaces in question, which we term \emph{tri-Coble surfaces}, are blow-ups of \(\P^2\) at \(12\) points which have contractions down to three different Coble surfaces.  The automorphisms arise as compositions of lifts of Bertini involutions from certain degree \(1\) weak del Pezzo surfaces.
\end{abstract}

\maketitle

\section{Introduction}

Suppose that \(X\) is a projective surface over an algebraically closed field \(K\) and that \(\phi : X \to X\) is an automorphism of \(X\).
When \(K = \mathbb C\), a theorem of Gromov and Yomdin asserts that \(\phi\) has positive topological entropy if and only if the spectral radius of \(\phi^\ast : N^1(X) \to N^1(X)\) is greater than \(1\), where \(N^1(X)\) denotes  the (finite dimensional) real vector space of divisors on \(X\) modulo numerical equivalence.
In a mild abuse of notation, for an arbitrary algebraically closed field \(K\) we will say that an automorphism \(\phi : X \to X\) has positive entropy if \(\phi^\ast : N^1(X) \to N^1(X)\) has spectral radius greater than \(1\).

Rational surfaces have proved to be an especially compelling source of examples of such automorphisms: although we do not attempt to provide an exhaustive bibliography, some representative constructions can be found in~\cite{bkfirst,coble,mcmullen,uehara}.
McMullen asked whether any rational surface \(X\) admitting a positive entropy automorphism must have a pluri-anticanonical curve, i.e.\ a curve belonging to some linear system \(\abs{-mK_X}\)~\cite[Question, pg.\ 87]{mcmullen}.  Since any automorphism \(f : X \to X\) preserves these linear systems, if one of them is nonempty then \(f\) must have invariant curves.

Bedford and Kim gave an elegant construction answering this question in the negative~\cite{bknocurves}.  Considering the family of birational maps \(f_{a,b} : \P^2 \rat \P^2\) defined in affine coordinates by \[(x,y) \mapsto \left( y, \frac{y+a}{x+b} \right),\] they show that by carefully choosing values for the parameters \(a\) and \(b\) and passing to a suitable blow-up, one obtains an automorphism of a rational surface with no periodic curves at all. 

In this note, we exhibit new positive entropy automorphisms of rational surfaces with no periodic curves.  We hope that these examples may still be of interest, as they have some new features: the examples are easily understood geometrically, exist in a positive dimensional family, and can be defined over the rational numbers. 

The strategy underlying the construction is straightforward.  Suppose that \(\bp = \set{p_1,\ldots,p_r}\) and \(\bq = \set{q_1,\ldots,q_s}\) are two configurations of points in \(\P^2\), and let \(S_{\bp}\) and \(S_{\bq}\)  denote the corresponding blow-ups.
Suppose too that both these surfaces admit nontrivial automorphisms, say \(\phi_\bp : S_\bp \to S_\bp\) and \(\phi_\bq : S_\bq \to S_\bq\).  If \(\phi_\bp\) fixes every point of \(\bq \setminus \bp\), and \(\phi_\bq\) fixes every point of \(\bp \setminus \bq\), then both automorphisms \(\phi_\bp\) and \(\phi_\bq\) lift to automorphisms of the common resolution \(S_{\bp  \bq}\).  Even if \(\phi_\bp\) and \(\phi_\bq\) each has an invariant curve, there is no reason  to expect that the composition \(\phi_\bp \circ \phi_\bq\) will fix either of these curves, let alone any other.

The difficulty lies in finding such configurations: for two automorphisms to each fix the base points of the other requires that these configurations be quite special (at least over \(\mathbb C\); we observe in \S\ref{poschar} that finding such configurations of points over the fields \(\Fpb\) is essentially trivial).

We ultimately employ this approach using not two but three sets of points, \(\bp\), \(\bq\), and \(\br\).  These will all be \(8\)-tuples, but with six points common to all three.  Each of the three blow-ups \(S_\bp\), \(S_\bq\), and \(S_\br\) is a weak del Pezzo surface of degree \(1\), and the three configurations are chosen so that the corresponding Bertini involutions all lift to the common model \(S_{\bp\bq\br}\).  Although the composition of any two of the involutions has invariant curves, we show that the composition of all three has none.

Before delving into the details, we give a quick description of the automorphism.  First, we construct the rational surface \(X\):

\begin{enumerate}
\item Choose three smooth quadrics \(Q_1\), \(Q_2\), and \(Q_3\) in \(\P^3\) so that any pair \(Q_i\) and \(Q_j\) are tangent at two points.  This determines three pairs of points, \(\bp\), \(\bq\), and \(\br\).  (Such configurations can easily be visualized; see Figure~\ref{quadricsfig}.)
\item Choose a cubic surface \(S \subset \P^3\) which passes through all six points, and is tangent to both of the quadrics passing through each.  Such  cubics surely exist, as there is a \(19\)-dimensional family of cubics and the tangency requirements impose only \(6 \times 3 = 18\) conditions.
\item Let \(X\) be the blow-up of \(S\) at the six points of tangency.
\end{enumerate}

\begin{figure}[hbt]
\begin{center}
  \includegraphics[scale=0.2]{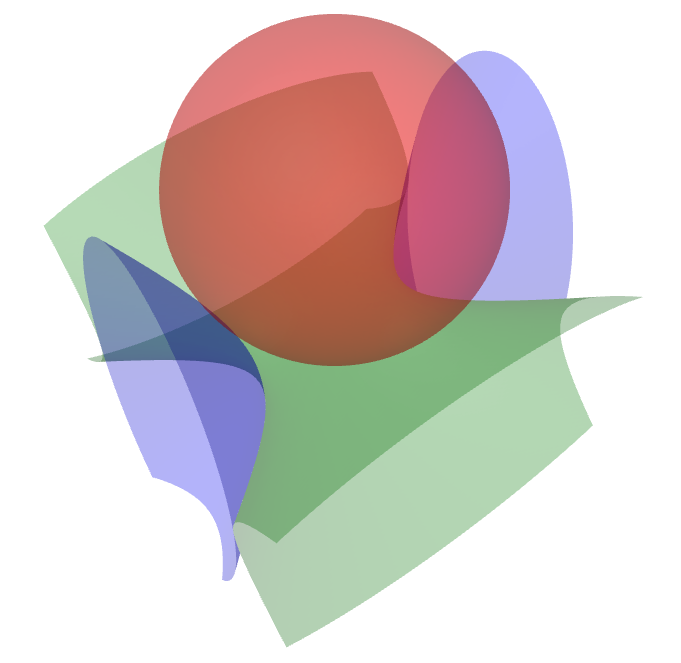}
\end{center}
\caption{A configuration of three tangent quadrics.   In this case, two of the quadrics are tangent along an entire curve.}
\label{quadricsfig}
\end{figure}

Now, to the three pairs \(\bp\), \(\bq\), and \(\br\), we associate involutions \(\tau_\bp\), \(\tau_\bq\), and \(\tau_\br\) as follows.

\begin{enumerate}
\setcounter{enumi}{3}
\item Given a general point \(z\) on \(S\), let \(\Pi \subset \P^3\) denote the plane through \(z\) and the two points of \(\bp = \set{p_1,p_2}\).
\item The intersection \(C = S \cap \Pi\) is a smooth genus \(1\) curve in \(\Pi\), passing through \(p_1\), \(p_2\), and \(z\).
\item There exists a unique conic \(\gamma \subset \Pi\) which is tangent to \(C\) at \(p_1\) and \(p_2\) and passes through \(z\).
\item \(\tau_\bp(z)\) is defined to be the residual sixth point of intersection of \(\gamma\) with \(C\).
\item The maps \(\tau_\bq\) and \(\tau_\br\) are defined analogously.
\end{enumerate}

Our main result is then:
\begin{theorem}
Let \(S\) be a cubic surface containing three pairs of points \(\bp\), \(\bq\), and \(\br\) as above.  Let \(X = S_{\bp\bq\br}\) be the blow-up of the cubic at \(6\) points.  For general parameter choices, the maps \(\tau_\bp\), \(\tau_\bq\), and \(\tau_\br\) all lift to biregular involutions of \(X\), and the composition \(\tau_\bp \circ \tau_\bq \circ \tau_\br\) is an automorphism of positive entropy which has no periodic curves.
\end{theorem}

\section{Coble surfaces and the Bertini involution}

We begin by recalling some classical geometry surrounding weak del Pezzo surfaces of degree \(1\), the Bertini involution, and Coble rational surfaces.  A reference for most of the results in this section is \cite[Chapter 8]{dolgachevclassical}.

A weak del Pezzo surface of degree \(k\) is a surface for which \(-K_S\) is big and nef and \((-K_S)^2 = k\); such surfaces exist for \(1 \leq k \leq 9\).  Blowing up \(9-k\) general points in \(\P^2\) yields examples, with \(-K_S\) ample; the weaker big and nef condition allows points in mildly degenerate configurations, for example configurations with \(3\) collinear points but which are otherwise general. These surfaces were classically the subject of intense study, and we recall two facts.

\begin{proposition}[{\cite[8.3.2]{dolgachevclassical}}]
Suppose that \(S\) is a weak del Pezzo surface of degree \(1\).
Then
\begin{enumerate}
\item \(\abs{-K_S}\) is a pencil of genus \(1\)  curves with one basepoint and smooth general member;
\item \(\abs{-2K_S}\) is \(4\)-dimensional and basepoint-free, and the \(2\)-anticanonical map \(\phi_{\abs{-2K_S}} : S \to \P^3\) is generically 2-to-1, with image a quadric cone. 
\end{enumerate}
The Bertini involution \(\tau : S \to S\) is defined to be the covering involution associated to \(\phi_{\abs{-2K_S}}\), which extends to a biregular map.
\end{proposition}

The Bertini involution admits a simple description in terms of the pencil \(\abs{-K_S}\).  Given a general point \(z \in S\), there is a unique smooth genus \(1\) curve \(C \in \abs{-K_S}\) passing through \(z\).  Since \(-2K_S \cdot C = 2\), there is a unique point \(z'\) on \(C\) for which \((-2K_S) \vert_{C} \otimes \cO_{C}(-z-z')\)
is trivial in \(\Pic^0(C)\), and so every element of \(\abs{-2K_S}\) passing through \(z\) also passes through \(z'\). This \(z'\) is the image of \(z\) under the Bertini involution.  As a result, we get a convenient characterization of the fixed points of Bertini involution:

\begin{lemma}
\label{bertinifixed}
Suppose that \(S\) is a weak del Pezzo surface of degree \(1\) and that \(z\) is a point which lies on a smooth curve \(C \in \abs{-K_S}\).  Then \(z\) is fixed by \(\tau\) if and only if
\((-2K_S)\vert_{C} \otimes \cO_{C}(-2z) = 0\) in \(\Pic^0(C)\).
\end{lemma}

A degree \(1\) del Pezzo can be obtained by blowing up \(2\) suitable points on a cubic surface, and it is easy to characterize when such two-point blow-ups are weak del Pezzo:

\begin{lemma}
Suppose that \(S \subset \P^3\) is a smooth cubic surface and that \(S_\bp = S\) is the blow-up of \(S\) at two smooth points \(\bp = \set{p_1,p_2}\).  Then \(S_\bp\) is a weak del Pezzo surface if and only if the line between \(p_1\) and \(p_2\) is not contained in \(S\).
\end{lemma}

\begin{proof}
Let \(E_1\) and \(E_2\) be the exceptional divisors of \(\pi : \Bl_\bp \P^3 \to \P^3\), and let \(H = \pi^\ast \cO_{\P^3}(1)\).  Then \(-K_S = (H -  E_1 - E_2) \vert_S\).  The divisor \(H - E_1 -E_2\) on \(\Bl_\bp \P^3\) is not nef, since it has intersection \(-1\) with the strict transform of the line through \(p_1\) and \(p_2\).  However, this is the only curve with which it has negative intersection, and so \((H-E_1-E_2)\vert_S\) is nef as long as this line is not contained in \(S\).   Since \((-K_S)^2 = 1\), we conclude that \(-K_S\) is big.
\end{proof}

Note that if \(S\) is the blow-up of a cubic in \(\P^3\) at two points, \(\abs{-2K_S}\) consists of quadric surfaces which are tangent to the cubic at each blown up point.  This gives rise to a convenient description of the Bertini involution on such blow-ups~\cite[pg.\ 128]{baker}.  Given a point \(z \in S\), there is \(3\)-dimensional family of quadric surfaces which are tangent to \(S\) at the two points of \(\bp\) and pass though the point \(z\).  This web of quadrics has a unique basepoint on \(S\), which coincides with \(\tau(z)\).
According to Lemma~\ref{bertinifixed}, a point \(z\) is fixed by \(\tau\) if in the plane \(\Pi = \Pi_{p_1 p_2 z}\), there exists a plane conic \(C \subset \Pi\)  which is tangent to \(S \cap \Pi\) at the three points \(p_1\), \(p_2\), and \(z\).

Weak del Pezzo surfaces never admit automorphisms of positive entropy, and it is necessary to look at rational surfaces obtained by blowing up additional points.  Central to our analysis are the \emph{Coble surfaces}.

\begin{definition}
A \emph{Coble surface} \(S\) is a smooth rational surface for which \(\abs{-K_S}\) is empty but \(\abs{-2K_S}\) is not.  A \emph{simple Coble surface} is a Coble surface for which \(\abs{-2K_S}\) is represented by a smooth rational curve.
\end{definition}

An application of the adjunction theorem shows that on a simple Coble surface, the rational curve \(C \in \abs{-2K_S}\) has self-intersection \(-4\). Such \(S\) can be obtained, for example, by blowing up the nodes of an irreducible, rational plane sextic with exactly ten nodes~\cite{coble}.

As in the del Pezzo case, rather than viewing a Coble surface as the blow-up of \(\P^2\) at a set of \(10\) points, we may regard it as a blow-up of a cubic surface at \(4\) special points.  Indeed, suppose that \(S\) is a cubic surface and that \(\bp\) is a quadruple of points on \(S\).  If there exists a conic \(Q\) which is tangent to \(S\) at the points \(\bp\) and such that \(Q \cap S\) is smooth, then \(S_\bp = \Bl_{\bp} S\) is a simple Coble surface, with \(Q \cap S \in \abs{-2K_{S_\bp}}\).

To carry out the strategy outlined in the induction, we would need a cubic surface \(S\) with two pairs of points \(\bp = \set{p_1,p_2}\) and \(\bq = \set{q_1,q_2}\) such that \(\tau_\bp\) fixes \(\bq\) and \(\tau_\bq\) fixes \(\bp\).
If this is the case, then both \(\tau_\bp\) and \(\tau_\bq\) lift to the blow-up \(S_{\bp\bq}\).
One might expect that \(\tau_\bp \circ \tau_\bq\) has no invariant curves, but the next result shows that this is too much to hope for.

\begin{theorem}
\label{blowupiscoble}
Suppose that \(S\) is a weak del Pezzo surface of degree \(3\)
and that \(\bp\) and \(\bq\) are two disjoint pairs of points on \(S\) such that \(S_\bp\) and \(S_\bq\) are both weak del Pezzo surfaces of degree \(1\).  Suppose further that two non-degeneracy conditions are satisfied:
\begin{enumerate}
\item[(N1)] The element of \(\abs{-K_S}\) through any three of the four points is smooth and irreducible, and these curves are distinct.
\item[(N2)] The tangent plane to \(S\) at any of the four points does not pass through any other.
\end{enumerate}
Then \(S_{\bp \bq}\) is a Coble surface if and only if
\(\tau_\bp\) fixes each point of \(\bq\) and \(\tau_\bq\) fixes each point of \(\bp\).
\end{theorem}

Before giving the proof, we record a simple geometric lemma.

\begin{lemma}
\label{quadrics}
Suppose that \(p_0,p_1,p_2,p_3\) are four non-coplanar points in \(\P^3\), and for \(0 \leq i \leq 3\), let \(\Pi_i\) denote the plane passing through the three points other than \(p_i\).
Suppose that there exist four smooth conics \(C_i \subset \Pi_i\) such that:
\begin{enumerate}
\item Each \(C_i\) passes through the three points \(p_j\) which lie on \(\Pi_i\);
\item At the point \(p_i\), we have \(\dim(T_{p_i} C_j + T_{p_i} C_k + T_{p_i} C_\ell) = 2\).  
\end{enumerate}
Then there exists a quadric \(Q \subset \P^3\) so that \(C_i = Q \cap \Pi_i\).
\end{lemma}

\begin{proof}
We may choose coordinates \([X_0,X_1,X_2,X_3]\) on \(\P^3\) so that the points are the four standard coordinate points. Then the plane \(\Pi_i\) is defined by \(X_i = 0\).

That the conics each pass through three of these points and lie in a plane \(X_i = 0\) means that they are given by equations of the form
\begin{align*}
F_0 &= a_{12} X_1 X_2 + a_{13} X_1 X_3 + a_{23} X_2 X_3 = 0, \\
F_1 &= b_{02} X_0 X_2 + b_{03} X_0 X_3 + b_{23} X_2 X_3 = 0, \\
F_2 &= c_{01} X_0 X_1 + c_{03} X_0 X_3 + c_{13} X_1 X_3 = 0, \\
F_3 &= d_{01} X_0 X_1 + d_{02} X_0 X_2 + d_{12} X_1 X_2 = 0
\end{align*}

To show that there exists a quadric \(Q\) as claimed, we must show that we may replace each \(F_i\) by a suitable multiple and arrange that any monomial \(X_i X_j\) has the same coefficient in both of the equations in which it appears.  The quadric \(Q\) is then defined by the polynomial obtained by combining all of these monomials.

We next work out the conditions on the coefficients given by the assumed tangency.  First consider the point \(p_0\).  In the affine chart on \(\P^3\) given by \(X_0 = 1\), this point is the origin, and the three quadrics are given by
\begin{align*}
F_1 &= b_{02}  X_2 + b_{03}  X_3 + b_{23} X_2 X_3 = 0 \\
F_2 &= c_{01}  X_1 + c_{03} X_3 + c_{13} X_1 X_3 = 0 \\
F_3 &= d_{01}  X_1 + d_{02} X_2 + d_{12} X_1 X_2 = 0
\end{align*}

Tangent vectors to these curves at the origin in \(\A^3\) are then given by
\[
(0,b_{03},-b_{02}), \quad (c_{03}, 0, -c_{01} ), \quad (d_{02}, -d_{01}, 0).
\]
The vectors are coplanar if the determinant of the matrix with these as rows vanishes, and and analogous computations at the other three coordinate points yield the four conditions
\begin{align*}
  b_{02} c_{03} d_{01} - b_{03} c_{01} d_{02} &= 0,&
  a_{12} c_{13} d_{01} - a_{13} c_{01} d_{12} &= 0, \\
  a_{12} b_{23} d_{02} - a_{23} b_{02} d_{12} &= 0,&
  a_{13} b_{23} c_{03} - a_{23} b_{03} c_{13} &= 0.
\end{align*}
That the conics are smooth implies that none of the coefficients vanish, and so after multiplying the equations by constants we may assume that the coefficients on the \(X_0 X_1\) and \(X_2 X_3\) terms are already equal, so that \(c_{01} = d_{01} = 1\) and \(a_{23} = b_{23} = 1\).  Our system of equations then becomes:
\begin{align*}
  b_{02} c_{03}  - b_{03}  d_{02} &= 0&
  a_{12} c_{13}  - a_{13}  d_{12} &= 0 \\
  a_{12}  d_{02} -  b_{02} d_{12} &= 0&
  a_{13}  c_{03} -  b_{03} c_{13} &= 0.
\end{align*}
This shows that
\[
\frac{a_{12}}{d_{12}} = \frac{b_{02}}{d_{02}} = \frac{a_{13}}{c_{13}} = \frac{b_{03}}{c_{03}}.
\]
Multiplying the equations \(F_2\) and \(F_3\) by this common value, we obtain multiples of the defining equations which have all corresponding coefficients equal.
\end{proof}

\begin{proof}[{Proof of Theorem~\ref{blowupiscoble}}]

The surface \(S\) can be mapped to \(\P^3\)  by the anticanonical map.  This map contracts the \((-2)\)-curves on \(S\), and the image is a cubic surface with du Val singularities.  The assumption that \(S_\bp\) and \(S_\bq\) are weak del Pezzo implies that none of the blown up points lies on a \((-2)\)-curve, so the four points of \(\bp\) and \(\bq\) are all mapped to smooth points.  The distinctness assumption in (N1) implies that no four of these points have coplanar image.

One direction of the proof is simple.  Suppose that \(S_{\bp\bq}\) is a Coble surface.  Then there exists a quadric \(Q\) which is tangent to \(S\) at the four points \(\bp \cup \bq\).
According to our description of the Bertini involution on a \(2\)-point blow-up of the cubic, the point \(q_1\) is fixed by \(\tau_\bp\) if and only if there exists a conic \(C\) in the plane \(\Pi = \Pi_{\bp q_1}\) such that \(C\) is double on the smooth cubic \(S \cap \Pi\) at these three points.  But there certainly exists such a conic: we can simply take \(C = Q \cap \Pi\), which is tangent at the points since \(Q\) and \(S\) are.  The other fixed point conditions follow in the same way.

Suppose instead that \(\tau_\bp\) fixes the points of \(\bq\) and \(\tau_\bq\) fixes the points of \(\bp\).  Let \(\Pi\) denote the plane through \(\bp\) and \(q_1\).   There exists a plane conic \(C \subset \Pi\) which is double on the plane cubic \(S \cap \Pi\) at \(p_1\), \(p_2\), and \(q_1\).

We claim that \(C\) must be smooth.  If not, then \(C\) is a union of two lines.  There are two possibilities: either (i) \(C\) is the double of a line \(L\) passing through both \(p_1\) and \(p_2\) or (ii) one of the lines \(L\) is tangent to \(S \cap C\) at \(p_1\).  Case (i) is ruled out by (N1), since this would mean that \(q_1\) lies on the line between \(p_1\) and \(p_2\), so that the four points are coplanar.  Case (ii) is ruled by (N2): if \(L\) passes through \(p_2\) or \(q\), this would mean that the tangent plane to \(S\) at \(p_1\) passes through \(p_2\) or \(q\).  If \(L\) misses both these points, then the second line \(L'\) must be tangent to \(S \cap \Pi\) at both \(p_2\) and \(q\), which again contradicts (N2).

Making the same argument for other triples of points,  we conclude that in any plane \(\Pi\) through three of the four points in \(\bp \cup \bq\), there exists a conic smooth \(C\) tangent to \(S \cap \Pi\) at these three points. Three such conics pass through each of the points in \(\bp \cup \bq\), and at each point the three conics have coplanar tangent directions, since all the tangents are contained in the tangent plane to \(S\).

It follows from Lemma~\ref{quadrics} that there is in fact a quadric surface \(Q \subset \P^3\) which is tangent to \(S\) at each of the four points, which shows that \(\abs{-2K_{S_{\bp\bq}}}\) is nonempty.  Since (N1) implies that the four points are not coplanar, \(\abs{-K_{S_{\bp\bq}}}\) is empty, and so the \(S_{\bp\bq}\) is a Coble surface.
\end{proof}

Blowing up two pairs of points on a cubic as in Theorem~\ref{blowupiscoble} yields a Coble surface, and every automorphism has an invariant curve.  We thus proceed to blow up a third pair of points, leading to the following definition.

\begin{definition}
\label{tricobledef}
Let \(S\) be a weak del Pezzo surface of degree \(3\) and suppose that there exist three disjoint pairs of points \(\bp\), \(\bq\), and \(\br\) on \(S\) such that:
\begin{enumerate}
\item[(T1)] \(S_\bp\), \(S_\bq\), and \(S_\br\) are weak del Pezzo surfaces of degree \(1\), with corresponding Bertini involutions \(\tau_\bp\), \(\tau_\bq\) and \(\tau_\br\);
\item[(T2)] \(\tau_\bp\) fixes the points of \(\bq \cup \br\), 
 \(\tau_\bq\) fixes the points of \(\bp \cup \br\), and
 \(\tau_\br\) fixes the points of \(\bp \cup \bq\).
\item[(T3)] Each of the \(4\)-tuples \(\bp \cup \bq\), \(\bp \cup \br\), and \(\bp \cup \bq\) satisfies the nondegeneracy conditions (N1) and (N2).
\end{enumerate}
Then we term the blow-up \(X = S_{\bp  \bq  \br}\) a \emph{tri-Coble surface}.  Notice that \(X\) may be contracted to each of \(S_{\bp  \bq}\), \(S_{\bp  \br}\), and \(S_{\bq  \br}\), which are all Coble surfaces according to Lemma~\ref{blowupiscoble}.  If any one of these three is a simple Coble surface, then we call \(X\) a \emph{simple tri-Coble surface}.
\end{definition}

Suppose that \(S\) is a cubic surface in \(\P^3\).  For \(\tau_\bp\) to fix the points of \(\bq\) and \(\tau_\bq\) to fix the points of \(\bp\) means that there is a quadric \(Q_1\) which is tangent to \(S\) at the four points of \(\bp \cup \bq\).  Similarly, there must exist a quadric \(Q_2\) tangent to \(S\) at \(\bp \cup \br\), and a quadric \(Q_3\) tangent to \(S\) at \(\bq \cup \br\).

\section{No invariant curves}

We will mostly be interested in the composition \(\phi = \tau_\bp \circ \tau_\bq \circ \tau_\br\).  The bi-anticanonical curve on \(S_{\bp  \bq}\) is  invariant under both \(\tau_\bp\) and \(\tau_\bq\), but there are no obvious curves invariant under all three of these involutions, and so it seems reasonable to expect that \(\phi\) does not have any invariant curves at all.

\begin{theorem}
Suppose that \(X = S_{\bp  \bq  \br}\) is a tri-Coble surface.  Then each of the involutions \(\tau_\bp\), \(\tau_\bq\), and \(\tau_\br\) lifts to a biregular involution of \(X\).  Let \(\phi = \tau_\bp \circ \tau_\bq \circ \tau_\br\) be the composition.  Then \(\phi\) is an automorphism of positive entropy.  If \(X\) is a simple tri-Coble surface, then \(\phi^\ast\) does not fix any pseudoeffective class.  In particular, \(\phi\) has no periodic curves.
\end{theorem}

\begin{proof}

First, observe that if \(C\) is a \(\phi\)-periodic curve, then \(\bigcup_n \phi^n(C)\) is a \(\phi\)-invariant (albeit reducible) curve, and so it suffices to show that there is no invariant curve.

The action of the Bertini involution on \(N^1(S)\) for a degree \(1\) weak del Pezzo surface was known classically~\cite{dolgachevclassical}.  If \(S\) is presented as a blow-up of \(\P^2\) at eight points with exceptional divisors \(E_1,\ldots,E_8\), then with respect to the basis \(H\), \(E_1\),\dots,\(E_8\):
\[
\tau^\ast = \left(\begin{array}{rrrrrrrrr}
17 & 6 & 6 & 6 & 6 & 6 & 6 & 6 & 6 \\
-6 & -3 & -2 & -2 & -2 & -2 & -2 & -2 & -2 \\
-6 & -2 & -3 & -2 & -2 & -2 & -2 & -2 & -2 \\
-6 & -2 & -2 & -3 & -2 & -2 & -2 & -2 & -2 \\
-6 & -2 & -2 & -2 & -3 & -2 & -2 & -2 & -2 \\
-6 & -2 & -2 & -2 & -2 & -3 & -2 & -2 & -2 \\
-6 & -2 & -2 & -2 & -2 & -2 & -3 & -2 & -2 \\
-6 & -2 & -2 & -2 & -2 & -2 & -2 & -3 & -2 \\
-6 & -2 & -2 & -2 & -2 & -2 & -2 & -2 & -3
\end{array}\right).
\]

The induced action of \(\tau_\bp\) on a tri-Coble surface is obtained by appending a \(4 \times 4\) identity matrix to the matrix \(\tau^\ast\), corresponding to the \(4\) exceptional divisors above the points \(\bq \cup \br\) which are invariant under \(\tau_\bp\).  The matrices for the other involutions are computed analogously, and multiplying all three together we obtain
\[
\phi^\ast = \left(\begin{array}{rrrrrrrrrrrrr}
377 & 126 & 126 & 126 & 126 & 126 & 126 & 150 & 150 & 30 & 30 & 6 & 6 \\
-126 & -43 & -42 & -42 & -42 & -42 & -42 & -50 & -50 & -10 & -10 & -2 & -2 \\
-126 & -42 & -43 & -42 & -42 & -42 & -42 & -50 & -50 & -10 & -10 & -2 & -2 \\
-126 & -42 & -42 & -43 & -42 & -42 & -42 & -50 & -50 & -10 & -10 & -2 & -2 \\
-126 & -42 & -42 & -42 & -43 & -42 & -42 & -50 & -50 & -10 & -10 & -2 & -2 \\
-126 & -42 & -42 & -42 & -42 & -43 & -42 & -50 & -50 & -10 & -10 & -2 & -2 \\
-126 & -42 & -42 & -42 & -42 & -42 & -43 & -50 & -50 & -10 & -10 & -2 & -2 \\
-6 & -2 & -2 & -2 & -2 & -2 & -2 & -3 & -2 & 0 & 0 & 0 & 0 \\
-6 & -2 & -2 & -2 & -2 & -2 & -2 & -2 & -3 & 0 & 0 & 0 & 0 \\
-30 & -10 & -10 & -10 & -10 & -10 & -10 & -12 & -12 & -3 & -2 & 0 & 0 \\
-30 & -10 & -10 & -10 & -10 & -10 & -10 & -12 & -12 & -2 & -3 & 0 & 0 \\
-150 & -50 & -50 & -50 & -50 & -50 & -50 & -60 & -60 & -12 & -12 & -3 & -2 \\
-150 & -50 & -50 & -50 & -50 & -50 & -50 & -60 & -60 & -12 & -12 & -2 & -3
\end{array}\right).
\]

A direct calculation shows that the characteristic polynomial is
\[
\chi_{\phi^\ast}(t) = (t-1)(t+1)^{10}(t^2-110t+1),
\]
and so the first dynamical degree is \(\lambda_1(\phi) = 55+12\sqrt{21} \approx 109.99\) and \(\phi\) has positive entropy.

Moreover, the \(1\)-eigenspace of \(\phi^\ast\) is \(1\)-dimensional, spanned by the canonical class \(K_X\).  Consequently, the only possible \(\phi\)-invariant curve would be pluri-anticanonical, and to prove that \(\tau\) admits no periodic curve, it suffices to show that \(\abs{-mK_X}\) is not effective for any \(m > 0\).   This is done in Corollary~\ref{nopluriantis} below.
\end{proof}

In fact, we will show that the anticanonical class on a simple tri-Coble surface \(X\) is not even pseudoeffective, (i.e.\ numerically a limit of effective classes), which implies that \(\abs{-mK_X }\) is empty for every \(m>0\). This is straightforward: the Coble surface \(S_{\bp\bq}\) has anti-bicanonical class represented by an irreducible curve of negative self-intersection.  A anti-bicanonical curve on \(X = S_{\bp\bq\br}\) would correspond to a member of \(\abs{-2K_{S_{\bp\bq}}}\) with nodes at the two points of \(\br\).  But this linear system contains only a single, smooth, rigid curve: it can not be deformed to have nodes at the points \(\br\).  This already shows that \(\abs{-2K_X}\) is not effective, and the fact that it is not pseudoeffective is an easy extension of the argument.

\begin{lemma}
\label{surfacemain}
Suppose that \(X\) is a smooth projective surface containing an irreducible curve \(C_1\) with \(C_1^2 < 0\). Let \(C_2\) be another curve satisfying \(C_1 \cdot C_2 \geq 0\).  Then for any \(\epsilon > 0\), the class \(C_1 - \epsilon C_2\) is not pseudoeffective.
\end{lemma}

\begin{proof}
Let \(A\) be an ample class on \(X\).  We show first that any effective representative of a class \(C_1 -\epsilon C_2 + \delta A\) for small \(\delta\) must have \(C_1\) in its support with multiplicity close to \(1\).  Indeed, suppose that \(C_1 + \epsilon A + \delta A \equiv r C_1 + F\), where \(F\) is an effective divisor whose support does not contain \(C_1\).  Then
\(F = (1-r)C_1 - \epsilon C_2 + \delta A\).  Intersecting both sides with \(C_1\), we obtain
\begin{align*}
  F \cdot C_1 &= (1-r)C_1^2 - \epsilon C_1 \cdot C_2 + \delta A \cdot C_1 \\
  0 &\leq  (1-r)C_1^2 + \delta A \cdot C_1 \\
  1-r &\leq -\delta \frac{A \cdot C_1}{C_1^2}.  
\end{align*}

The second line follows from the fact that \(F \cdot C_1 \geq 0\) while the third requires \(C_1^2 < 0\).  Now compute
\begin{align*}
  F \cdot A &= (1-r)C_1 \cdot A  - \epsilon C_2 \cdot A + \delta A^2 \\
            &\leq \left(  -\delta \frac{A \cdot C_1}{C_1^2} \right) C_1 \cdot A - \epsilon C_2 \cdot A + \delta A^2 \\
  &= \delta \left(A^2  - \frac{A \cdot C_1}{C_1^2} \right) - \epsilon C_2 \cdot A.
\end{align*}

If \(\epsilon\) is fixed and \(\delta <  \left(A^2  - \frac{A \cdot C_1}{C_1^2} \right)^{-1}(A \cdot C_2)\) is taken sufficiently small, the right side is evidently negative, so that \(F \cdot A < 0\), which is impossible if \(F\) is effective.  It follows that \(C_1 - \epsilon C_2 + \delta A\) is not effective for sufficiently small \(\delta\), so that \(C_1 - \epsilon C_2\) is not pseudoeffective.
\end{proof}

\begin{lemma}
\label{blowuppseff}
Suppose that \(S\) is a smooth surface and \(C \subset S\) is a smooth, irreducible curve with \(C^2 < 0\).  Let \(p\) any point of \(S\), and let \(\pi : S^\prime \to S\) be the blow-up of \(S\) at \(p\), with exceptional divisor \(E\).  Then the class \(\pi^\ast C - \epsilon E\) is not pseudoeffective for any \(\epsilon > 1\).
\end{lemma}

\begin{proof}
The strict transform \(C'\) of \(C\) on \(S^\prime\) has has class \(\pi^\ast C - b E\), where \(b\) is either \(0\) or \(1\) (depending on whether \(p\) lies on \(C\)).  Then \(C'\) is a smooth, irreducible curve with \((C')^2 < 0\), and the claim follows from Lemma~\ref{surfacemain} taking \(C_1 = C'\) and \(C_2 = E\).
\end{proof}

\begin{corollary}
\label{nopluriantis}
Let \(X\) be a simple tri-Coble surface.  Then \(-K_X\) is not pseudoeffective, and every linear system \(\abs{-mK_X}\) is empty.
\end{corollary}

\begin{proof}
The simplicity hypothesis means that \(X\) is obtained by blowing up two points on \(S_{\bp\bq}\), where \(S_{\bp\bq}\) has a unique irreducible rational curve in \(\abs{-2K_{S_{\bp\bq}}}\).  From the adjunction theorem, this curve is smooth of self-intersection \(-4\), and since \(-2K_X = \pi^\ast(-2K_{S_{\bp\bq}}) - 2(E_1 + E_2)\), it follows from Lemma~\ref{blowuppseff} that this class is not pseudoeffective.
\end{proof}

\section{The existence of tri-Coble surfaces}
\label{existence}

There is still one piece missing: we must prove that tri-Coble surfaces actually exist.  The proof is by direct construction. By definition, a tri-Coble surface is a \(6\)-point blow-up of  a smooth degree-\(3\) weak del Pezzo surface \(S\).  To construct a tri-Coble surface, we will take \(S\) to be  a particular cubic surface in \(\P^3\) and then blow up three pairs of points on \(S\), which arise as the tangency points \(S\) with three quadric surfaces, as described following Definition~\ref{tricobledef}.

To actually find such configurations, it is helpful to invert our perspective.  Rather than fixing a cubic surface \(S \subset \P^3\) and searching for six points \(\bp,\bq,\br\) such that there exist three quadrics each tangent at four of the six, we begin with three quadrics \(Q_1,Q_2,Q_3 \subset \P^3\) such that each pair of quadrics are tangent at two points.  Only after fixing the quadrics and the six tangency points do we construct the cubic surface.  For a cubic surface to pass through a point and have a given tangent plane there imposes \(3\) conditions.  Since the space of cubic surfaces has dimension \(19\), we expect that given six points and prescribed tangent planes at those points, there should exist a \(19-6 \times 3 = 1\)-dimensional pencil of satisfactory cubic surfaces.  

We must then check several non-degeneracy conditions to ensure that our surface is actually a simple tri-Coble surface:
\begin{enumerate}
\item \(S\) is smooth and irreducible;
\item the blow-ups \(S_\bp\), \(S_\bq\), \(S_\br\) are weak del Pezzo;
\item no four of the six points are coplanar;
\item the intersection of \(S\) with any plane through any three of the six points is a smooth cubic;
\item the tangent plane to \(S\) at any point does not pass through any other points;
\item the intersections \(Q_i \cap S_{\bp\bq\br}\) are smooth and irreducible.
\end{enumerate}

(Here (1) and (2) show that (T1) is satisfied.  (3) and (4) together imply that (N1) holds as required by (T2), while (5) checks (N2).  At last, (6) implies that the surface is actually a \emph{simple} tri-Coble surface, so that one of the three blow-down Coble surfaces has a smooth bi-anticanonical curve.)

With a bit of computer-aided experimentation, this strategy readily yields examples where the quadrics and the tangency points are all defined over \(\Q\).  Consider the following three quadric surfaces in \(\P^3\), with coordinates \((w:x:y:z)\):
\begin{align*}
Q_1 &: 59 w^{2} + x^{2} + y^{2} - 20 w z + z^{2} = 0,\\
Q_2 &: - 9w^{2} - x^{2} + 9y^{2} +  z^{2} = 0,\\
Q_3 &: - 9w^{2} - x^{2} + 9y^{2} - 6 z^{2} = 0,
\end{align*}
as well as the six points
\begin{align*}
p_1 &= ( 1 : 4 : 0 : 5 ), &  q_1 &= ( 1 : 0 : 5 : 6 ), & r_1 &= ( 12 : -15 : 13 : 0 ), \\
p_2 &= ( 1 : -4 : 0 : 5 ), & q_2 &= ( 1 : 0 : -5 : 6 ), & r_2 &= ( -3 : 12 : 5 : 0 ).
\end{align*}

The following table gives the tangent spaces \(T_{p_i} Q_j\) at the six points, using the coordinates from the dual \(\P^3\).  An ``\ex'' in a column indicates that the surface does not pass through the point.
\begin{align*}
\begin{array}{c|c|c|c|}
& Q_1 & Q_2 & Q_3 \\\hline
p_1 & ( -9 : -4 : 0 : 5 ) & ( -9 : -4 : 0 : 5 ) & \ex \\\hline
p_2 & ( -9 : 4 : 0 : 5 ) & ( -9 : 4 : 0 : 5 ) & \ex \\\hline
q_1 & ( 1 : 0 : -5 : 4 ) & \ex & ( 1 : 0 : -5 : 4 ) \\\hline
q_2 & ( 1 : 0 : 5 : 4 ) & \ex & ( 1 : 0 : 5 : 4 ) \\\hline
r_1 & \ex & ( -36 : 5 : 39 : 0 ) & ( -36 : 5 : 39 : 0 ) \\\hline
r_2 & \ex & ( 9 : -4 : 15 : 0 ) & ( 9 : -4 : 15 : 0 ) \\\hline
\end{array}
\end{align*}

An examination of the table shows that these quadrics satisfy the required pointwise tangency conditions.
It is then an exercise in linear algebra to write down a cubic surface with the prescribed tangent planes at all six of these points.  There is in fact a \(1\)-dimensional family of such cubics, with one such surface \(S\) defined by the vanishing of the equation
\begin{align*}
  F &= 9963 w^{3} + 56187 w^{2} x + 27707 w x^{2} + 3018 x^{3} + 12069 w^{2} y + 366 x^{2} y + 11457 w y^{2} \\ \tag{$\ast$} &\qquad  + 3213 x y^{2} + 351 y^{3} - 5358 w^{2} z - 11610 w x z - 3002 x^{2} z - 4140 w y z + 18 y^{2} z  \\ &\qquad- 7643 w z^{2} - 1857 x z^{2} + 111 y z^{2} + 38 z^{3}.
\end{align*}

\begin{remark}
An easy dimension count suggests that on a general cubic \(S\) it should be possible to find pairs \(\bp\), \(\bq\), and \(\br\) so that the blow-up \(S_{\bp\bq\br}\) is a tri-Coble surface.  However, actually finding such points, even on the Fermat cubic, leads to equations with no obvious solutions. The approach followed in this section seems to be more straightforward, even if it leads to a somewhat cumbersome cubic surface.
\end{remark}

To show \(S_{\bp\bq\br}\) is tri-Coble surface, it is still necessary to verify non-degeneracy conditions (1)-(6) ensuring that the automorphisms \(\tau_\bp\), \(\tau_\bq\), and \(\tau_\br\) all well-defined and that their composition has no periodic curves.
It seems unsurprising that these non-degeneracy conditions hold, but since the cubic \(S\) is constructed as the solution to an interpolation problem, it is difficult to directly control its geometry.

Although these verifications are tedious to carry out by hand, they are routine on a computer.  Notice that checking (2) does not actually necessitate the relatively difficult task of finding the \(27\) lines on \(S\): it is enough to check that the particular lines connecting pairs of points are not contained in \(S\). 
The nondegeneracy properties are checked for the specific cubic ($\ast$) in the attached Sage script.  Note that all of these are open conditions, and so the general cubic in the pencil constructed earlier also satisfies these conditions.

\section{Examples over \(\Fpb\)}
\label{poschar}

In general it seems difficult to find configurations \(\bp\) and \(\bq\) with positive-entropy automorphisms \(\phi_\bp\) and \(\phi_\bq\) for which \(\bq \setminus \bp\) is invariant under \(\phi_\bp\) and \(\bp \setminus \bq\) is invariant under \(\phi_\bq\).  However, we observe now that over the field \(k = \Fpb\), the algebraic closure of a finite field, essentially any configuration will do.   Most constructions of automorphisms of rational surfaces, such as those of Bedford--Kim~\cite{bkfirst} or McMullen~\cite{mcmullen}, work perfectly well over these fields, although one must exercise some care that the characteristic is large enough to ensure that the configurations \(\bp\) constructed there actually consist of sets of distinct points.

\begin{theorem}
\label{poschariterate}
Let \(\bp\) and \(\bq\) be two configurations of points in \(\P^2\) for which there exist automorphisms \(\phi_\bp : S_\bp \to S_\bp\) and \(\phi_\bq : S_\bq \to S_\bq\).  Then there exist positive integers \(m\) and \(n\) so that \(\phi_\bp^m \circ \phi_\bq^n\) lifts to an automorphism of \(S_{\bp \bq}\).
\end{theorem}

\begin{proof}
\(\phi_\bp : S_\bp \to S_\bp\) is defined over an algebraically closed field \(\Fpb\).  There exists a finite field \(\F_q\) such that all the points of \(\bp\) and \(\bq\) and the maps \(\phi_\bp\) and \(\phi_\bq\) are all defined over \(\F_q\).

Now, the number of \(\F_q\)-points on \(S_\bp\) is finite, and these points are permuted by \(\phi_\bp\), so some iterate \(\phi_\bp^m\) fixes all the points of \(\bq \setminus \bp\).  Similarly, an iterate \(\phi_\bq^n\) fixes all the points in \(\bp \setminus \bq\).  Then the composition \(\phi_\bp^m \circ \phi_\bq^n\) of these two iterates lifts to an automorphism of the blow-up \(S_{\bp\bq}\).
\end{proof}

Let \(S_q^{\text{all}}\) denote the blow-up of \(\P^2_{\Fpb}\) at all \(\F_q\)-points for some prime power \(q = p^s\).  The proof of Theorem~\ref{poschariterate} shows that if \(\bp\) is a configuration defined over \(\F_q\), then any automorphism \(\phi : S_\bp \to S_\bp\) has an iterate that lifts to an automorphism of \(S_q^{\text{all}}\), so it seems reasonable to expect that this group is quite large once \(q\) is sufficiently big.

\begin{question}
What can be said about the group \(\Aut(S_q^{\text{all}})\)?  Is it finitely generated?
\end{question}

\bibliographystyle{plain}
\bibliography{refs}

\section{Acknowledgements}

This work was supported by NSF grant DMS-1912476 and a Sloan Research Fellowship.  The \texttt{SageMath} system was invaluable for computations; in particular, the cubic given by ($\ast$) was found using the shortest vector algorithm provided by the \texttt{IntegerLattice} package.

\end{document}